\documentclass[11pt]{amsart}

\usepackage[T1]{fontenc}
\usepackage[utf8]{inputenc}
\usepackage{lmodern}
\usepackage{amsmath,amssymb,amsthm,mathtools}
\usepackage{booktabs}
\usepackage{enumitem}
\usepackage{microtype}
\usepackage{array}
\usepackage{longtable}
\usepackage{hyperref}
\usepackage[nameinlink,capitalise]{cleveref}

\hypersetup{
  colorlinks=true,
  linkcolor=blue,
  citecolor=blue,
  urlcolor=blue,
  pdftitle={Sharp Broken-Power Lorentz Estimates for Fractional Powers of Radial Schrodinger Operators with Inverse-Square Asymptotics},
  pdfauthor={Haochen Liu, Qinghao Yu, Hongyan Zhou}
}

\newtheorem{theorem}{Theorem}[section]
\newtheorem{proposition}[theorem]{Proposition}
\newtheorem{lemma}[theorem]{Lemma}
\newtheorem{corollary}[theorem]{Corollary}
\newtheorem{remark}[theorem]{Remark}

\newcommand{\R}{\mathbb R}

\newcommand{\Z}{\mathbb Z}
\newcommand{\one}{\mathbf 1}
\newcommand{\dd}{\,\mathrm d}
\newcommand{\supp}{\operatorname{supp}}
\newcommand{\dist}{\operatorname{dist}}

\newcommand{\Lp}[2]{L^{#1,#2}}

\newcommand{\muU}{\mu_U}
\newcommand{\Kothe}{K\"othe}

\title[Sharp broken-power Lorentz estimates]
{Sharp Broken-Power Lorentz Estimates for Fractional Powers of Radial Schr\"odinger Operators with Inverse-Square Asymptotics}

\author{Haochen Liu}
\address{School of Mathematics and Physics, Qingdao University of Science and Technology, Qingdao, China}
\email{liuhc@mails.qust.edu.cn}

\author{Qinghao Yu}
\address{School of Mathematics and Physics, Qingdao University of Science and Technology, Qingdao, China}
\email{19707725295@163.com}

\author{Hongyan Zhou}
\address{School of Mathematics and Physics, Qingdao University of Science and Technology, Qingdao, China}
\email{zhouhongyan@qust.edu.cn}
\thanks{Corresponding author: Hongyan Zhou.}
\date{July 13, 2026}


\begin{document}

\begin{abstract}
Let $H=-\Delta+V(|x|)$ be a nonnegative radial Schr\"odinger operator on $\R^d$, $d\ge2$, whose positive harmonic function has two power regimes
\[
U(r)\simeq r^{-\sigma_0}\quad(0<r\le1),
\qquad
U(r)\simeq r^{-\sigma_\infty}\quad(r\ge1),
\]
with $-d/2<\sigma_0,\sigma_\infty<d/2$, and assume the two-sided ground-state heat-kernel estimate of Ishige--Kabeya--Ouhabaz.  We first determine the maximal open range in which the kernel of $H^{-s/2}$ has the clean form
\[
K_s^H(x,y)\simeq |x-y|^{s-d}
\frac{U(|x|)U(|y|)}
{U(|x|+|x-y|)U(|y|+|x-y|)}.
\]
It is
\[
0<s<\min\{d,d-2\sigma_0,d-2\sigma_\infty\}.
\]
Within this range we give a complete necessary-and-sufficient classification of the broken-power estimate
\[
\begin{gathered}
\|w_{-\beta_0,-\beta_\infty}H^{-s/2}f\|_{L^{q,v}}
\lesssim
\|w_{\alpha_0,\alpha_\infty}f\|_{L^{p,u}},\\
1<p,q<\infty,\qquad 1\le u,v\le\infty.
\end{gathered}
\]
The theorem includes all signed ground-state exponents, the full upper-triangle range $q<p$, all one-sided weight equalities, both scale equalities, and simultaneous endpoint corners.  Scale equality is governed by the Lorentz condition $u\le v$ even when $q<p$, while an input or output power endpoint requires respectively $u=1$ or $v=\infty$; at a same-side power/scale corner the only admissible pair is $(u,v)=(1,\infty)$.  The proof combines a maximal clean-kernel argument, exact local Lorentz--HLS estimates, rank-one endpoint analysis, geometric annular sequence spaces, a signed triangular matrix theorem, and a nine-block decomposition.
\end{abstract}

\maketitle

\section{Introduction}

Fractional integration begins with the Hardy--Littlewood--Sobolev theorem and its weighted Stein--Weiss extension; see \cite{HardyLittlewood1928,Sobolev1938,SteinWeiss1958,AdamsHedberg1996,Grafakos2014}.  In the classical Euclidean setting the kernel is homogeneous, so scaling provides a single linear relation among the exponents.  Power weights and more general two-weight problems have been studied from many points of view, including weighted convolution and interpolation inequalities and the work of Muckenhoupt--Wheeden, Sawyer, and Sawyer--Wheeden \cite{Kerman1983,CaffarelliKohnNirenberg1984,MuckenhouptWheeden1974,Sawyer1988,SawyerWheeden1992}.  In the upper triangle $q<p$, nonlinear potentials and sequential testing reveal an additional summability mechanism \cite{CascanteOrtegaVerbitsky2004,Tanaka2014,HanninenHytonenLi2016}.  Related sharp power-weight theories occur for radial functions and Hankel-type transforms \cite{Duoandikoetxea2013,GorbachevIvanov2019,NowakStempak2017}.

For a Schr\"odinger operator, fractional powers are no longer convolution operators.  Heat-kernel estimates provide a natural replacement for Fourier analysis; see, for instance, \cite{Davies1989,Ouhabaz2005,Shen1995,BongioanniHarboureSalinas2008,HuZahle2009}.  The inverse-square scale is critical for elliptic and dispersive equations.  For the exactly homogeneous operator
\[
\mathcal L_a=-\Delta+a|x|^{-2},
\]
precise kernel estimates and weighted Sobolev theory are available in \cite{KMMVZZ2018,BuiEtAl2017}; a sharp two-weight endpoint analysis for that homogeneous model was recently given in \cite{LiuYuZhou2026}.  The present paper treats a different phenomenon: the potential is only asymptotically inverse square and the positive ground state may have one power at the origin and another at infinity.

Our heat-kernel input is the theorem of Ishige, Kabeya, and Ouhabaz \cite{IKO2017}.  Under a radial inverse-square framework and the hypothesis $U^2\in A_2$, they obtain a positive harmonic function $U$ and two-sided Gaussian estimates after a ground-state transform.  The $A_2$ geometry is consistent with the weighted elliptic theory of Fabes--Kenig--Serapioni \cite{FabesKenigSerapioni1982} and with the classical Muckenhoupt framework \cite{Muckenhoupt1972}.  Lorentz estimates for inverse-square heat semigroups have also been investigated in \cite{IshigeTateishi2020a,IshigeTateishi2021}.  Here the heat kernel is integrated in the spectral parameter to obtain an explicit fractional kernel, after which the full Lorentz two-weight range is computed.

The first main point is that the clean formula has a genuinely sharp uniform range.  If
\[
U(r)\simeq r^{-\sigma_0}\quad(r\le1),
\qquad
U(r)\simeq r^{-\sigma_\infty}\quad(r\ge1),
\]
then
\[
K_s^H(x,y)\simeq |x-y|^{s-d}
\frac{U(|x|)U(|y|)}
{U(|x|+|x-y|)U(|y|+|x-y|)}
\]
holds uniformly exactly in the open range
\[
0<s<s_{\rm clean}:=\min\{d,d-2\sigma_0,d-2\sigma_\infty\}.
\]
The restriction $s<d$ comes from a fixed transition annulus, $s<d-2\sigma_0$ from small comparable radii, and $s<d-2\sigma_\infty$ from the large-time tail.  At each boundary an additional logarithm or a saturation effect destroys first-scale comparability.

The second main point is the exact Lorentz classification.  For the broken power
\[
w_{a_0,a_\infty}(x)=|x|^{a_0}\one_{\{|x|\le1\}}+|x|^{a_\infty}\one_{\{|x|>1\}},
\]
we classify
\begin{equation}\label{eq:intro-estimate}
\|w_{-\beta_0,-\beta_\infty}H^{-s/2}f\|_{L^{q,v}}
\le C
\|w_{\alpha_0,\alpha_\infty}f\|_{L^{p,u}}
\end{equation}
for all $1<p,q<\infty$ and $1\le u,v\le\infty$.  Set
\[
S=s+\frac dq-\frac dp.
\]
The power region is described by seven nonnegative margins: a local margin $S$, four one-sided origin/infinity margins, and two scale margins.  Unlike a global homogeneous theorem, there is no single scaling equality.  Instead one obtains the sandwich
\[
\alpha_0+\beta_0\le S\le\alpha_\infty+\beta_\infty.
\]
The fine indices decide which equalities survive.  The most notable conclusions are:
\begin{enumerate}[label=\textup{(\roman*)},leftmargin=2.4em]
\item every scale equality requires $u\le v$, even when $S>0$;
\item scale equality can occur for $q<p$ after a Lorentz refinement, although it is excluded in strong $L^p\to L^q$ form;
\item an input-critical power requires $u=1$, and an output-critical power requires $v=\infty$;
\item when a scale equality meets a power endpoint at the same geometric end, the only admissible pair is $(u,v)=(1,\infty)$.
\end{enumerate}

The proof is not a direct application of an abstract two-weight criterion.  Such criteria provide important context \cite{Sawyer1988,Kairema2013,CascanteOrtegaVerbitsky2004,Tanaka2014,HanninenHytonenLi2016}, but the present two-regime kernel permits an explicit calculation.  A partition into origin, transition, and infinity produces nine blocks.  Near the diagonal one needs an exact truncated Riesz theorem in Lorentz spaces.  Separated transition blocks are rank one.  Deep at either end, separated annuli are encoded by a triangular sequence operator, while comparable annuli retain the local singularity and require an operator-level argument.  A key technical warning is that a general Lorentz norm cannot be scalarized as an $\ell^u$ norm of arbitrary local Lorentz pieces.  We use geometric packet discretization only where it is valid, conditional expectation for separated shell rectangles, and a separate comparable-radius lemma.

Lorentz spaces and convolution inequalities go back to Lorentz, O'Neil, and Hunt \cite{Lorentz1950,ONeil1963,Hunt1966}; standard interpolation references include \cite{BerghLofstrom1976,BennettSharpley1988}.  We prove the exact Lorentz estimates needed here directly from rearrangements, so no unverified endpoint interpolation or weak-space density assertion is used.  In particular, the operator on $L^{p,\infty}$ is defined by positive monotone truncation.

The paper is organized as follows.  In \cref{sec:setting} we formulate the heat-kernel framework and state the main results.  The clean fractional kernel and its maximal range are proved in \cref{sec:kernel}.  Lorentz preliminaries and the local HLS theorem appear in \cref{sec:lorentz}.  Rank-one and annular sequence lemmas are established in \cref{sec:annular}, followed by the exact triangular matrix theorem in \cref{sec:triangular}.  Necessity is proved in \cref{sec:necessity}.  Comparable-radius blocks are treated in \cref{sec:comparable}, and the nine-block sufficiency proof is completed in \cref{sec:sufficiency}.  The final section records extension, duality, strong-space reductions, and model cases.

\section{Framework and main results}\label{sec:setting}

Let $d\ge2$, and let $H=-\Delta+V(|x|)$ be a nonnegative self-adjoint Schr\"odinger operator on $L^2(\R^d)$, realized by the Friedrichs or form construction.  We assume that there is a positive radial $H$-harmonic function $U$ satisfying
\begin{equation}\label{eq:U-powers}
U(r)\simeq r^{-\sigma_0}\quad(0<r\le1),
\qquad
U(r)\simeq r^{-\sigma_\infty}\quad(r\ge1),
\end{equation}
where
\begin{equation}\label{eq:sigma-A2}
-\frac d2<\sigma_0,\sigma_\infty<\frac d2,
\qquad U^2\in A_2(\R^d).
\end{equation}
Define
\[
\dd\muU(x)=U(|x|)^2\dd x.
\]
We further assume the two-sided heat-kernel estimate
\begin{equation}\label{eq:heat-kernel-assumption}
p_t^H(x,y)\simeq
\frac{U(|x|)U(|y|)}
{\sqrt{\muU(B(x,\sqrt t))\muU(B(y,\sqrt t))}}
\exp\!\left(-c\frac{|x-y|^2}{t}\right),
\end{equation}
where the constants and the Gaussian parameters may differ in the two directions of comparison.  The radial class in \cite{IKO2017} supplies \eqref{eq:U-powers}--\eqref{eq:heat-kernel-assumption} in its pure-power $A_2$ branch.  We formulate the paper at the level of these consequences, thereby separating the mapping theorem from the detailed classification of radial potentials.

For $s>0$, define formally
\begin{equation}\label{eq:Mellin-def}
H^{-s/2}=\frac1{\Gamma(s/2)}\int_0^\infty t^{s/2-1}e^{-tH}\dd t.
\end{equation}
The next theorem identifies the range in which this operator has a clean pointwise kernel.

\begin{theorem}[Maximal clean fractional-kernel range]\label{thm:clean-kernel}
Assume \eqref{eq:U-powers}--\eqref{eq:heat-kernel-assumption} and put
\begin{equation}\label{eq:s-clean}
s_{\rm clean}=\min\{d,d-2\sigma_0,d-2\sigma_\infty\}.
\end{equation}
For every
\[
0<s<s_{\rm clean},
\]
the spectral projection $E_H(\{0\})$ vanishes and $H^{-s/2}$ has a positive kernel satisfying
\begin{equation}\label{eq:clean-kernel-main}
K_s^H(x,y)\simeq |x-y|^{s-d}
\frac{U(|x|)U(|y|)}
{U(|x|+|x-y|)U(|y|+|x-y|)},
\qquad x\ne y.
\end{equation}
The interval is maximal for a uniform comparison of the form \eqref{eq:clean-kernel-main}: at $s=d$, $s=d-2\sigma_0$, or $s=d-2\sigma_\infty$ whenever the corresponding boundary is active, one obtains a logarithmic or saturation obstruction.
\end{theorem}

Fix henceforth $0<s<s_{\rm clean}$.  For $a_0,a_\infty\in\R$, write
\begin{equation}\label{eq:broken-weight}
w_{a_0,a_\infty}(x)=
\begin{cases}
|x|^{a_0},&|x|\le1,\\
|x|^{a_\infty},&|x|>1.
\end{cases}
\end{equation}
Let $1<p,q<\infty$ and $1\le u,v\le\infty$.  Set
\begin{equation}\label{eq:S-def}
S=s+\frac dq-\frac dp
\end{equation}
and define
\begin{equation}\label{eq:margins}
\begin{aligned}
A_0&=\frac d{p'}-\alpha_0-\sigma_0,
& B_0&=\frac dq-\beta_0-\sigma_0,\\
C_0&=S-\alpha_0-\beta_0,\\
A_\infty&=\alpha_\infty-\left(s+\sigma_\infty-\frac dp\right),
& B_\infty&=\beta_\infty-\left(s+\sigma_\infty-\frac d{q'}\right),\\
C_\infty&=\alpha_\infty+\beta_\infty-S.
\end{aligned}
\end{equation}
We also use
\begin{equation}\label{eq:gamma-def}
\gamma_j=d-s-2\sigma_j>0,
\qquad j\in\{0,\infty\}.
\end{equation}
A direct calculation gives
\begin{equation}\label{eq:margin-identities}
A_0+B_0=C_0+\gamma_0,
\qquad
A_\infty+B_\infty=C_\infty+\gamma_\infty.
\end{equation}
Thus at a scale equality $C_j=0$, the two same-side power endpoints cannot both vanish.

\begin{theorem}[Sharp broken-power Lorentz classification]\label{thm:main}
Under the assumptions above, the estimate
\begin{equation}\label{eq:main-estimate}
\left\|w_{-\beta_0,-\beta_\infty}H^{-s/2}f\right\|_{L^{q,v}(\R^d)}
\le C
\left\|w_{\alpha_0,\alpha_\infty}f\right\|_{L^{p,u}(\R^d)}
\end{equation}
holds for all measurable $f$ if and only if
\begin{equation}\label{eq:power-all}
S,A_0,B_0,C_0,A_\infty,B_\infty,C_\infty\ge0
\end{equation}
and all applicable fine-index rules below are satisfied:
\begin{align}
&S=0 \quad\Longrightarrow\quad u\le v,\label{eq:fine-S}\\
&A_0=0\ \text{or}\ A_\infty=0
\quad\Longrightarrow\quad u=1,\label{eq:fine-A}\\
&B_0=0\ \text{or}\ B_\infty=0
\quad\Longrightarrow\quad v=\infty,\label{eq:fine-B}\\
&C_0=0\ \text{or}\ C_\infty=0
\quad\Longrightarrow\quad u\le v,\label{eq:fine-C}\\
&C_j=0\ \text{and}\ (A_j=0\ \text{or}\ B_j=0)
\quad\Longrightarrow\quad (u,v)=(1,\infty),
\quad j=0,\infty.\label{eq:fine-corner}
\end{align}
There are no further fine-index restrictions.
\end{theorem}

Equivalently, the power conditions are
\begin{equation}\label{eq:power-expanded}
S\ge0,
\qquad
\alpha_0+\sigma_0\le\frac d{p'},
\qquad
\beta_0+\sigma_0\le\frac dq,
\end{equation}
\begin{equation}\label{eq:power-expanded-infty}
\alpha_\infty\ge s+\sigma_\infty-\frac dp,
\qquad
\beta_\infty\ge s+\sigma_\infty-\frac d{q'},
\end{equation}
and
\begin{equation}\label{eq:sandwich}
\alpha_0+\beta_0\le S\le\alpha_\infty+\beta_\infty.
\end{equation}

\begin{remark}[Decision rule]\label{rem:decision}
If all seven margins are positive, every pair $1\le u,v\le\infty$ is admissible.  If only local or scale equalities occur, the exact condition is $u\le v$.  An isolated input endpoint requires $u=1$; an isolated output endpoint requires $v=\infty$.  Endpoint requirements at different geometric ends simply accumulate.  At a same-side power/scale corner, \eqref{eq:fine-corner} overrides the separate rules and leaves only $(1,\infty)$.
\end{remark}

\begin{table}[t]
\centering
\caption{Exact fine-index rule for one deep end $j\in\{0,\infty\}$.  The row $A_j=B_j=0$ is impossible under \eqref{eq:margin-identities}.}
\label{tab:truth}
\begin{tabular}{cccc}
\toprule
$A_j$ & $B_j$ & $C_j$ & Fine-index condition\\
\midrule
$>0$ & $>0$ & $>0$ & none\\
$=0$ & $>0$ & $>0$ & $u=1$\\
$>0$ & $=0$ & $>0$ & $v=\infty$\\
$>0$ & $>0$ & $=0$ & $u\le v$\\
$=0$ & $>0$ & $=0$ & $(u,v)=(1,\infty)$\\
$>0$ & $=0$ & $=0$ & $(u,v)=(1,\infty)$\\
\bottomrule
\end{tabular}
\end{table}

\begin{corollary}[Strong spaces]\label{cor:strong}
Set $u=p$ and $v=q$.  Then \eqref{eq:main-estimate} holds if and only if
\[
S\ge0,
\]
\[
\alpha_0+\sigma_0<\frac d{p'},
\qquad
\beta_0+\sigma_0<\frac dq,
\]
\[
\alpha_\infty>s+\sigma_\infty-\frac dp,
\qquad
\beta_\infty>s+\sigma_\infty-\frac d{q'},
\]
and
\[
\alpha_0+\beta_0\le S\le\alpha_\infty+\beta_\infty
\quad\text{if }p\le q,
\]
whereas both scale inequalities are strict if $q<p$.
\end{corollary}

\begin{corollary}[Free Lorentz HLS]\label{cor:free}
If $H=-\Delta$, all four weight exponents vanish, and $0<s<d$, then the power conditions force
\[
\frac1q=\frac1p-\frac sd.
\]
The estimate $I_s:\Lp p u\to\Lp q v$ holds exactly when $u\le v$.
\end{corollary}

\begin{corollary}[Globally homogeneous model]\label{cor:homogeneous}
Suppose $U(r)=r^{-\sigma}$ and
$\alpha_0=\alpha_\infty=\alpha$, $\beta_0=\beta_\infty=\beta$.  Then the scale conditions force
\[
\alpha+\beta=S.
\]
Put
\[
A=\frac d{p'}-\alpha-\sigma,
\qquad
B=\frac dq-\beta-\sigma.
\]
The estimate holds exactly when $S,A,B\ge0$, with $u\le v$ if $A,B>0$, and only $(u,v)=(1,\infty)$ if $A=0$ or $B=0$.
\end{corollary}

\section{Ground-state geometry and the fractional kernel}\label{sec:kernel}

We write $X=|x|$, $Y=|y|$, and $r=|x-y|$ throughout this section.

\begin{lemma}[Weighted volume]\label{lem:volume}
For all $x\in\R^d$ and $\rho>0$,
\begin{equation}\label{eq:volume}
\muU(B(x,\rho))\simeq \rho^dU(|x|+\rho)^2.
\end{equation}
\end{lemma}

\begin{proof}
If $\rho\le X/2$, then $|z|\simeq X$ on $B(x,\rho)$.  The two-power comparability in \eqref{eq:U-powers}, including the bounded transition annulus, gives
\[
\muU(B(x,\rho))\simeq \rho^dU(X)^2\simeq \rho^dU(X+\rho)^2.
\]
If $X<2\rho$, then
\[
B(0,\rho)\subset B(x,3\rho),
\qquad
B(x,\rho)\subset B(0,3\rho),
\]
and a standard finite-overlap comparison reduces the assertion to centered balls.  For $R\le1$,
\[
\muU(B(0,R))\simeq \int_0^R t^{d-1-2\sigma_0}\dd t
\simeq R^{d-2\sigma_0}=R^dU(R)^2,
\]
where $\sigma_0<d/2$ is essential.  The same argument for $R\ge1$, splitting at $1$, gives
\[
\muU(B(0,R))\simeq R^{d-2\sigma_\infty}=R^dU(R)^2.
\]
The finitely many transition configurations are absorbed into the comparison constant.
\end{proof}

Combining \eqref{eq:heat-kernel-assumption} and \cref{lem:volume}, we obtain
\begin{equation}\label{eq:heat-kernel-euclidean}
p_t^H(x,y)\simeq
\frac{U(X)U(Y)}{t^{d/2}U(X+\sqrt t)U(Y+\sqrt t)}
\exp\!\left(-c\frac{r^2}{t}\right).
\end{equation}
Consequently, after $t=\rho^2$,
\begin{equation}\label{eq:fractional-integral-rho}
K_s^H(x,y)\simeq
U(X)U(Y)\int_0^\infty
\frac{\rho^{s-d}}{U(X+\rho)U(Y+\rho)}
\exp\!\left(-c\frac{r^2}{\rho^2}\right)\frac{\dd\rho}{\rho}.
\end{equation}
The two directions of comparison may use different $c>0$; this is harmless in all arguments below.

\begin{lemma}[Polynomial tempering]\label{lem:tempering}
There are constants $C,M>0$ such that for all $a,b>0$,
\begin{equation}\label{eq:tempering}
\frac{U(a)}{U(b)}\le C\max\left\{\left(\frac ab\right)^M,\left(\frac ba\right)^M\right\}.
\end{equation}
\end{lemma}

\begin{proof}
Inside either pure regime the quotient is a fixed power up to constants.  If $a$ and $b$ lie in different regimes, insert the scale $1$.  Positivity and continuity on a compact transition interval supply the remaining uniform constants.
\end{proof}

\begin{lemma}[Dyadic first-term dominance]\label{lem:first-term}
Let
\[
F_{X,Y}(\rho)=\frac{\rho^{s-d}}{U(X+\rho)U(Y+\rho)}.
\]
Assume $0<s<s_{\rm clean}$.  Then, uniformly in $X,Y\ge0$ and $r>0$ with $|X-Y|\le r$,
\begin{equation}\label{eq:first-term}
\int_0^\infty F_{X,Y}(\rho)e^{-cr^2/\rho^2}\frac{\dd\rho}{\rho}
\simeq F_{X,Y}(r).
\end{equation}
\end{lemma}

\begin{proof}
For the lower bound, restrict to $r\le\rho\le2r$.  The Gaussian is bounded below, while \cref{lem:tempering} and bounded-scale comparability give
\[
F_{X,Y}(\rho)\simeq F_{X,Y}(r)
\]
on this interval.

For the upper bound on $0<\rho<r$, split into
\[
2^{-n-1}r<\rho\le2^{-n}r,
\qquad n\ge0.
\]
By \cref{lem:tempering}, the quotient of the $U$ factors is at most $C2^{Mn}$, whereas
\[
e^{-cr^2/\rho^2}\le e^{-c4^n}.
\]
Hence the sum of all sub-$r$ contributions is bounded by a convergent series times $F_{X,Y}(r)$.

For $\rho\ge r$, note that
\begin{equation}\label{eq:XY-comparable}
|X-Y|\le r\le\rho
\quad\Longrightarrow\quad
X+\rho\simeq Y+\rho.
\end{equation}
Set
\[
a_k=F_{X,Y}(2^kr),\qquad k\ge0.
\]
Apart from a bounded number of indices where one of $2^kr$, $X+2^kr$, $Y+2^kr$ is comparable to $1$, every maximal run lies in one of three regimes.  On such a run and for every $n\ge0$ remaining in the same regime,
\[
a_{k+n}\lesssim a_k2^{n(s-d)},
\]
or
\[
a_{k+n}\lesssim a_k2^{n(s-d+2\sigma_0)},
\]
or
\[
a_{k+n}\lesssim a_k2^{n(s-d+2\sigma_\infty)}.
\]
All three exponents are negative precisely under $s<s_{\rm clean}$.  Each run is therefore a decreasing geometric series, and the bounded transition sets can be absorbed into adjacent first terms.  Iterating through the finitely many possible regime changes gives
\[
\sum_{k\ge0}a_k\lesssim a_0.
\]
The integral over $\rho\ge r$ is comparable to this dyadic sum, proving the upper bound.
\end{proof}

\begin{proposition}[Fixed off-diagonal convergence]\label{prop:fixed-convergence}
For fixed $x\ne y$, the integral in \eqref{eq:fractional-integral-rho} is finite exactly when
\[
0<s<d-2\sigma_\infty.
\]
At $s=d-2\sigma_\infty$ the large-$\rho$ divergence is logarithmic.
\end{proposition}

\begin{proof}
The Gaussian suppresses every power as $\rho\downarrow0$.  As $\rho\to\infty$,
\[
\frac{\rho^{s-d}}{U(X+\rho)U(Y+\rho)}\simeq
\rho^{s-d+2\sigma_\infty}.
\]
The claim follows by integration against $\dd\rho/\rho$.
\end{proof}

\begin{lemma}[Absence of a zero spectral projection]\label{lem:zero-mode}
Under \eqref{eq:U-powers}--\eqref{eq:heat-kernel-assumption}, one has $E_H(\{0\})=0$.
\end{lemma}

\begin{proof}
For $t\ge1$, \eqref{eq:heat-kernel-euclidean} gives
\[
p_t^H(x,y)\lesssim U(X)U(Y)t^{-(d-2\sigma_\infty)/2}.
\]
Since $\sigma_\infty<d/2$, the right-hand side tends to zero.  If $f,g$ are bounded and compactly supported away from the origin, direct integration yields
\[
\langle e^{-tH}f,g\rangle\longrightarrow0.
\]
By the spectral theorem, $e^{-tH}$ converges strongly to $E_H(\{0\})$ as $t\to\infty$.  The stated test functions are dense in $L^2(\R^d)$, so $E_H(\{0\})=0$.
\end{proof}

\begin{proof}[Proof of \cref{thm:clean-kernel}]
By \cref{lem:zero-mode}, the Mellin formula \eqref{eq:Mellin-def} represents $H^{-s/2}$ without a zero-mode contribution.  Formula \eqref{eq:fractional-integral-rho} and \cref{lem:first-term} give
\[
K_s^H(x,y)\simeq
U(X)U(Y)\frac{r^{s-d}}{U(X+r)U(Y+r)},
\]
which is \eqref{eq:clean-kernel-main}.

It remains to justify maximality.  First fix $X,Y$ in a compact annulus and let $r\downarrow0$.  On $r<\rho<1$ the denominator is bounded above and below, so the contribution is comparable to
\[
\int_r^1\rho^{s-d}\frac{\dd\rho}{\rho}.
\]
At $s=d$ this equals $\log(1/r)$, while the proposed first term is constant; for $s>d$ the scale-one endpoint dominates.  Thus $s<d$ is necessary.

Next take $X,Y,r$ mutually comparable and tending to zero.  On $r<\rho<1$ the contribution is comparable to
\[
U(X)U(Y)\int_r^1\rho^{s-d+2\sigma_0}\frac{\dd\rho}{\rho}.
\]
At $s=d-2\sigma_0$ an extra logarithm appears, and above this value the upper endpoint rather than $\rho\simeq r$ controls the integral.  Finally, \cref{prop:fixed-convergence} shows that $s<d-2\sigma_\infty$ is necessary.  This proves maximality of the stated uniform open range.
\end{proof}

\begin{remark}\label{rem:negative-sigma-clean}
If $\sigma_0\le0$, then $d-2\sigma_0\ge d$, so the origin restriction is redundant.  If $\sigma_\infty\le0$, then $d-2\sigma_\infty\ge d$, so the infinity restriction is also redundant once $s<d$.  In particular, if both exponents are negative, the clean formula holds exactly for $0<s<d$.
\end{remark}

\section{Lorentz preliminaries and local fractional integration}\label{sec:lorentz}

For a measurable function $h$, write
\[
\mu_h(\lambda)=|\{x:|h(x)|>\lambda\}|,
\qquad
h^*(t)=\inf\{\lambda>0:\mu_h(\lambda)\le t\}.
\]
For $1<r<\infty$ and $1\le a<\infty$, we use
\begin{equation}\label{eq:Lorentz-norm}
\|h\|_{L^{r,a}}
=\left(\int_0^\infty [t^{1/r}h^*(t)]^a\frac{\dd t}{t}\right)^{1/a},
\end{equation}
and
\begin{equation}\label{eq:weak-norm}
\|h\|_{L^{r,\infty}}=\sup_{t>0}t^{1/r}h^*(t).
\end{equation}
These conventions give $L^{r,r}=L^r$ with equivalent norms.  We use fine-index conjugates
\[
\frac1a+\frac1{a'}=1,
\qquad
1'=\infty,
\qquad
\infty'=1.
\]
The basic structural facts below are standard; see \cite{Lorentz1950,ONeil1963,Hunt1966,BennettSharpley1988}.

\begin{lemma}[Elementary Lorentz facts]\label{lem:Lorentz-facts}
Let $1<r<\infty$ and $1\le a,b\le\infty$.
\begin{enumerate}[label=\textup{(\roman*)}]
\item If $0<|E|<\infty$, then
\[
\|\one_E\|_{L^{r,a}}=
\begin{cases}
(r/a)^{1/a}|E|^{1/r},&a<\infty,\\
|E|^{1/r},&a=\infty.
\end{cases}
\]
\item At fixed primary exponent,
\[
L^{r,a}\hookrightarrow L^{r,b}
\quad\Longleftrightarrow\quad a\le b.
\]
\item On a finite-measure set, if $r_1>r_0$, then
\[
L^{r_1,a}\hookrightarrow L^{r_0,b}
\]
for arbitrary fine indices $a,b$.
\item The \Kothe{} associate space is
\[
(L^{r,a})'=L^{r',a'}.
\]
For $a=\infty$, this is an associate-space statement, not an identification of the full Banach dual.
\end{enumerate}
\end{lemma}

\begin{lemma}[Lorentz product inequality]\label{lem:product}
Suppose
\[
\frac1c=\frac1a+\frac1b<1,
\qquad
\frac1t=\frac1r+\frac1z\le1.
\]
Then
\begin{equation}\label{eq:product}
\|fg\|_{L^{c,t}}
\lesssim
\|f\|_{L^{a,r}}\|g\|_{L^{b,z}}.
\end{equation}
The same conclusion holds with any larger target fine index.
\end{lemma}

\begin{proof}
The rearrangement inequality
\[
(fg)^*(2\tau)\le f^*(\tau)g^*(\tau)
\]
gives
\[
\tau^{1/c}(fg)^*(2\tau)
\le [\tau^{1/a}f^*(\tau)][\tau^{1/b}g^*(\tau)].
\]
Apply H\"older's inequality on $(0,\infty)$ with measure $\dd\tau/\tau$.  The weak cases follow by the usual supremum interpretation, and a larger target fine index follows from \cref{lem:Lorentz-facts}(ii).
\end{proof}

We shall also use the Hardy--Littlewood maximal operator.  The rearrangement estimate
\begin{equation}\label{eq:max-rearrangement}
(Mh)^*(t)\lesssim h^{**}(t),
\qquad
h^{**}(t)=\frac1t\int_0^t h^*(r)\dd r,
\end{equation}
and the one-dimensional Hardy inequality imply
\begin{equation}\label{eq:max-Lorentz}
\|Mh\|_{L^{p,u}}\lesssim\|h\|_{L^{p,u}},
\qquad 1<p<\infty,
\quad 1\le u\le\infty.
\end{equation}

\subsection{The local truncated Riesz theorem}

Let $A_1,A_2$ be fixed nondegenerate bounded annuli.  Let $\chi$ be a bounded cutoff supported in
\[
A_2\times A_1\cap\{|x-y|<c_0\}
\]
and bounded below by a positive constant on two smaller overlapping annuli for all sufficiently small $|x-y|$.  Define
\begin{equation}\label{eq:local-Riesz}
\mathcal I_{s,\chi}g(x)=\int_{\R^d}\chi(x,y)|x-y|^{s-d}g(y)\dd y.
\end{equation}

\begin{theorem}[Exact local Lorentz theorem]\label{thm:local}
The map
\[
\mathcal I_{s,\chi}:L^{p,u}\longrightarrow L^{q,v}
\]
is bounded if and only if
\[
S=s+\frac dq-\frac dp\ge0,
\]
and, when $S=0$, one additionally has $u\le v$.  If $S>0$, every pair $u,v$ is allowed.
\end{theorem}

\begin{proof}
Assume first $S=0$, so
\begin{equation}\label{eq:critical-relation}
\frac1q=\frac1p-\frac sd.
\end{equation}
Splitting the kernel at the radius whose ball has measure $t$ and applying the Hardy--Littlewood rearrangement inequality gives
\begin{equation}\label{eq:Riesz-rearrangement}
(\mathcal I_{s,\chi}g)^*(t)
\lesssim
t^{s/d-1}\int_0^t g^*(r)\dd r
+
\int_t^{|A_1|}g^*(r)r^{s/d-1}\dd r.
\end{equation}
The second integral is interpreted as zero for $t>|A_1|$.  Put
\[
G(r)=r^{1/p}g^*(r).
\]
Multiply \eqref{eq:Riesz-rearrangement} by $t^{1/q}$ and set $t=e^\tau$, $r=e^\rho$.  Using \eqref{eq:critical-relation}, the two terms become one-sided convolutions of $G(e^\rho)$ against
\[
e^{-(1-1/p)(\tau-\rho)}\one_{\{\rho<\tau\}}
\quad\text{and}\quad
e^{-(1/q)(\rho-\tau)}\one_{\{\rho>\tau\}}.
\]
Both kernels belong to $L^1(\R)$.  Young's inequality on the logarithmic line therefore gives
\begin{equation}\label{eq:local-same-fine}
\|\mathcal I_{s,\chi}g\|_{L^{q,u}}
\lesssim \|g\|_{L^{p,u}}
\end{equation}
for every $1\le u\le\infty$.  The embedding $L^{q,u}\hookrightarrow L^{q,v}$ proves sufficiency when $u\le v$.

To prove necessity of $u\le v$, choose a point in the overlap and geometrically shrinking, mutually disjoint pairs of balls $E_k,F_k$ such that
\[
|E_k|\simeq|F_k|\simeq r_k^d,
\qquad
\dist(E_k,F_k)\simeq r_k,
\qquad
r_k=L^{-k},
\]
with $L$ sufficiently large.  For finitely supported $c=(c_k)$ put
\[
g=\sum_k c_kr_k^{-d/p}\one_{E_k}.
\]
Geometric rearrangement yields
\begin{equation}\label{eq:local-packet-in}
\|g\|_{L^{p,u}}\simeq\|c\|_{\ell^u}.
\end{equation}
On $F_k$, positivity and the local lower bound give
\[
\mathcal I_{s,\chi}g\gtrsim c_kr_k^{-d/q}.
\]
Hence
\begin{equation}\label{eq:local-packet-out}
\|\mathcal I_{s,\chi}g\|_{L^{q,v}}
\gtrsim\|c\|_{\ell^v}.
\end{equation}
Taking $c_k=1$ for $1\le k\le N$ forces $N^{1/v}\lesssim N^{1/u}$.

Assume next that $S>0$.  Since $s<d$,
\[
\frac Sd=\frac sd+\frac1q-\frac1p
<\frac1{p'}+\frac1q.
\]
Choose $a,b>0$ such that
\[
a+b=\frac Sd,
\qquad
a<\frac1{p'},
\qquad
b<\frac1q.
\]
Define
\[
\frac1{p_0}=\frac1p+a,
\qquad
\frac1{q_0}=\frac1q-b.
\]
Then $1<p_0<p$, $q<q_0<\infty$, and
\[
\frac1{q_0}=\frac1{p_0}-\frac sd.
\]
The critical estimate gives $L^{p_0,r}\to L^{q_0,r}$ for any fixed $r$.  Since both annuli have finite measure, strict primary-exponent embeddings yield
\[
L^{p,u}(A_1)\hookrightarrow L^{p_0,r}(A_1),
\qquad
L^{q_0,r}(A_2)\hookrightarrow L^{q,v}(A_2)
\]
for arbitrary $u,v$.

Finally, if $S<0$, take $g=\one_{B(z,\varepsilon)}$ in the overlap.  On a concentric ball of comparable radius,
\[
\mathcal I_{s,\chi}g\gtrsim\varepsilon^s.
\]
The norm ratio is bounded below by a constant times $\varepsilon^S$, which diverges as $\varepsilon\downarrow0$.
\end{proof}

\begin{corollary}[Global Lorentz HLS]\label{cor:global-HLS}
Let $0<\lambda<d$, $1<p<d/\lambda$, and $1/q=1/p-\lambda/d$.  Then
\[
I_\lambda:L^{p,u}(\R^d)\longrightarrow L^{q,v}(\R^d)
\]
is bounded exactly when $u\le v$.
\end{corollary}

\begin{proof}
The rearrangement estimate is the same as \eqref{eq:Riesz-rearrangement}, with $|A_1|$ replaced by $\infty$.  The logarithmic-line convolution proof gives $L^{p,u}\to L^{q,u}$, and the packet lower test proves necessity.
\end{proof}

\section{Rank-one endpoints and geometric annular sequences}\label{sec:annular}

\subsection{Rank-one operators}

\begin{lemma}[Exact rank-one criterion]\label{lem:rank-one}
Let $a,b$ be nonzero measurable functions and
\[
R_{b,a}g(x)=b(x)\int_{\R^d}a(y)g(y)\dd y.
\]
Then
\begin{equation}\label{eq:rank-one-criterion}
R_{b,a}:L^{p,u}\longrightarrow L^{q,v}
\quad\Longleftrightarrow\quad
a\in L^{p',u'},\quad b\in L^{q,v}.
\end{equation}
Moreover,
\[
\|R_{b,a}\|\simeq
\|a\|_{L^{p',u'}}\|b\|_{L^{q,v}}.
\]
\end{lemma}

\begin{proof}
Sufficiency follows from Lorentz H\"older and \cref{lem:Lorentz-facts}(iv).  For necessity, choose a sequence in the unit ball of $L^{p,u}$ that norms the positive functional $g\mapsto\int |a|g$.  Since $b\ne0$, boundedness forces that functional to belong to the \Kothe{} associate $L^{p',u'}$.  Taking one $g$ for which the functional is nonzero then forces $b\in L^{q,v}$.
\end{proof}

The four separated factors arising from the kernel are
\begin{equation}\label{eq:rank-factors}
\begin{aligned}
a_0(y)&=|y|^{-\alpha_0-\sigma_0}\one_{\{|y|<1/4\}},
&b_0(x)&=|x|^{-\beta_0-\sigma_0}\one_{\{|x|<1/4\}},\\
a_\infty(y)&=|y|^{s-d+\sigma_\infty-\alpha_\infty}\one_{\{|y|>4\}},
&b_\infty(x)&=|x|^{s-d+\sigma_\infty-\beta_\infty}\one_{\{|x|>4\}}.
\end{aligned}
\end{equation}
The power $|x|^{-d/r}$, either near the origin or at infinity, belongs to $L^{r,\infty}$ and to no $L^{r,a}$ with $a<\infty$.  Consequently,
\begin{equation}\label{eq:rank-endpoints}
A_0=0\text{ or }A_\infty=0\Longrightarrow u=1,
\qquad
B_0=0\text{ or }B_\infty=0\Longrightarrow v=\infty.
\end{equation}
For later necessity arguments, a critical input density truncated on $e^{-N}<|x|<e^{-1}$ has $L^{p,u}$ norm comparable to $N^{1/u}$, while its pairing with the critical associate density is comparable to $N$.  After normalization this leaves $N^{1-1/u}$ and forces $u=1$.  The output statement follows directly from the Lorentz norm of the critical power.

\subsection{Geometric shell packets}

Let $(E_k)$ be disjoint measurable shells on a half-line of indices, with
\begin{equation}\label{eq:geometric-measures}
|E_k|\simeq R_k^d,
\qquad
R_k=L^k,
\qquad L>1.
\end{equation}
For a scalar sequence $c=(c_k)$ define
\begin{equation}\label{eq:Jr}
J_rc=\sum_k c_k|E_k|^{-1/r}\one_{E_k}.
\end{equation}

\begin{lemma}[Geometric annular Lorentz norm]\label{lem:geometric-Lorentz}
For $1<r<\infty$ and $1\le a\le\infty$,
\begin{equation}\label{eq:geometric-equivalence}
\|J_rc\|_{L^{r,a}}\simeq\|c\|_{\ell^a}.
\end{equation}
\end{lemma}

\begin{proof}
Assume first $a<\infty$.  The distribution formula
\[
\|h\|_{L^{r,a}}^a
=r\int_0^\infty\lambda^{a-1}\mu_h(\lambda)^{a/r}\dd\lambda
\]
and \eqref{eq:Jr} give
\[
\mu_{J_rc}(\lambda)=
\sum_{\{k:|c_k||E_k|^{-1/r}>\lambda\}}|E_k|.
\]
Put $m_k=|E_k|$, $b_k=|c_k|m_k^{-1/r}$, and choose $n_k\in\Z$ with
$2^{n_k}\le b_k<2^{n_k+1}$.  Dyadic discretization in $\lambda$ yields
\begin{equation}\label{eq:dist-discrete}
\|J_rc\|_{L^{r,a}}^a
\simeq
\sum_{n\in\Z}2^{na}
\left(\sum_{k:n_k\ge n}m_k\right)^{a/r}.
\end{equation}
Because the $m_k$ form a geometric family, for every subset $F$ and every $\theta>0$,
\begin{equation}\label{eq:geometric-subset}
\left(\sum_{k\in F}m_k\right)^\theta
\simeq_\theta
\sum_{k\in F}m_k^\theta.
\end{equation}
Indeed, both sides are comparable to the $\theta$-power of the largest member of the geometric set.  Apply \eqref{eq:geometric-subset} with $\theta=a/r$ in \eqref{eq:dist-discrete}, interchange sums, and obtain
\[
\|J_rc\|_{L^{r,a}}^a
\simeq
\sum_km_k^{a/r}\sum_{n\le n_k}2^{na}
\simeq
\sum_km_k^{a/r}2^{n_ka}
\simeq
\sum_k|c_k|^a.
\]
For $a=\infty$, the distribution formula and geometric summability give
\[
\sup_{\lambda>0}\lambda\mu_{J_rc}(\lambda)^{1/r}
\simeq\sup_k|c_k|.
\]
\end{proof}

\begin{remark}[No arbitrary scalarization]\label{rem:no-scalarization}
The geometric hypothesis is essential.  In general,
\[
\|h\|_{L^{p,u}}
\not\simeq
\left\|\big(\|h\one_{E_k}\|_{L^{p,u}}\big)_k\right\|_{\ell^u}.
\]
For equal-measure subsets in $N$ different shells, the left side carries the primary multiplicity $N^{1/p}$, whereas the right side carries $N^{1/u}$.  We use \cref{lem:geometric-Lorentz} for normalized geometric packets and for shell-constant conditional expectations, never as a general local decomposition formula.
\end{remark}

\begin{lemma}[Diagonal multipliers]\label{lem:diagonal}
For a scalar sequence $D=(D_k)$, multiplication $M_D:\ell^u\to\ell^v$ is bounded exactly when
\begin{equation}\label{eq:diag-condition}
\begin{cases}
D\in\ell^\infty,&u\le v,\\
D\in\ell^\rho,\quad \displaystyle\frac1\rho=\frac1v-\frac1u,&v<u.
\end{cases}
\end{equation}
On the origin half-line $k\le0$, the multiplier $D_k=2^{kC}$ is bounded $\ell^u\to\ell^v$ exactly when either $C>0$, or $C=0$ and $u\le v$.
\end{lemma}

\begin{proof}
If $u\le v$, use $\ell^u\hookrightarrow\ell^v$ and point masses for necessity.  If $v<u$, H\"older proves sufficiency with $1/\rho=1/v-1/u$.  Necessity follows by testing finite truncations of the sequence $|D_k|^{\rho/u}$ with the appropriate sign.  The geometric half-line statement follows immediately.
\end{proof}

\section{The exact triangular annular theorem}\label{sec:triangular}

The separated part of a deep origin or infinity block is governed by a one-sided triangular matrix.  We formulate the origin version on
\[
\Z_-:=\{k\in\Z:k\le0\}.
\]
For $A,B,C\in\R$ and a finitely supported sequence $c$ define
\begin{equation}\label{eq:triangular-op}
(\mathsf H_{A,B,C}c)_k
=
2^{kC}\sum_{\ell\le k}2^{-A(k-\ell)}c_\ell
+
\sum_{\ell\ge k}2^{-B(\ell-k)}2^{\ell C}c_\ell.
\end{equation}

\begin{theorem}[Exact triangular sequence theorem]\label{thm:triangular}
Let $1\le u,v\le\infty$.  The operator
\[
\mathsf H_{A,B,C}:\ell^u(\Z_-)\longrightarrow\ell^v(\Z_-)
\]
is bounded if and only if $A,B,C\ge0$ and the applicable row of the following table holds:
\begin{center}
\begin{tabular}{cccc}
\toprule
$A$ & $B$ & $C$ & exact fine-index condition\\
\midrule
$>0$ & $>0$ & $>0$ & none\\
$=0$ & $>0$ & $>0$ & $u=1$\\
$>0$ & $=0$ & $>0$ & $v=\infty$\\
$=0$ & $=0$ & $>0$ & $(u,v)=(1,\infty)$\\
$>0$ & $>0$ & $=0$ & $u\le v$\\
$\ge0$ & $\ge0$ & $=0$ & $(u,v)=(1,\infty)$ if $AB=0$\\
\bottomrule
\end{tabular}
\end{center}
\end{theorem}

\begin{proof}
Suppose first $A,B>0$.  The two one-sided kernels $2^{-An}\one_{n\ge0}$ and $2^{-Bn}\one_{n\ge0}$ belong to $\ell^1$.  Thus the first term in \eqref{eq:triangular-op} is an $\ell^u$-bounded convolution followed by multiplication by $2^{kC}$, and the second is convolution applied to $(2^{kC}c_k)$.  The conclusion follows from \cref{lem:diagonal}: when $C>0$ the geometric multiplier maps every $\ell^u$ to every $\ell^v$, while at $C=0$ it maps exactly for $u\le v$.

If $A=0$, the first term contains the prefix operator
\begin{equation}\label{eq:prefix}
2^{kC}\sum_{\ell\le k}c_\ell.
\end{equation}
For $u=1$ this is bounded by $\|c\|_{\ell^1}2^{kC}$.  If $C>0$, the geometric output lies in every $\ell^v$; if $C=0$, it lies only in $\ell^\infty$.  The second term is harmless when $B>0$.  The case $B=0$ is dual: the tail operator maps into $\ell^\infty$, and finite target fine index is impossible at the endpoint.  Combining these observations proves sufficiency in every endpoint row, including the corners.

For necessity, a negative $C$ fails on point masses sent toward $-\infty$.  If $A<0$, take a point mass at a very negative index and evaluate the first sum at $k=0$; if $B<0$, take a point mass at $0$ and evaluate the second sum far to the left.  At $A=0$, a block of $N$ normalized coefficients makes a prefix sum grow like $N^{1-1/u}$, hence $u=1$.  At $B=0$, a point mass creates $N$ descendants of comparable size, hence $v=\infty$.  At $C=0$ with $A,B>0$, diagonal transmission of $N$ separated coefficients gives
\[
N^{1/v}\lesssim N^{1/u},
\]
so $u\le v$.  The same tests combine at the corners and prove the remaining necessity statements.
\end{proof}

We now identify the geometric parameters.  Let $E_k=\{2^{k-1}<|x|<2^{k+1}\}$ for $k\le-3$, and write
\[
g=w_{\alpha_0,\alpha_\infty}f.
\]
When $x\in E_k$, $y\in E_\ell$ and $\ell\le k-3$, the clean kernel and the two origin powers give, after normalization by $|E_k|^{1/q}$ and $|E_\ell|^{1/p}$,
\begin{equation}\label{eq:origin-coeff-one}
2^{kC_0}2^{-A_0(k-\ell)}.
\end{equation}
When $k\le\ell-3$, the normalized coefficient is
\begin{equation}\label{eq:origin-coeff-two}
2^{\ell C_0}2^{-B_0(\ell-k)}.
\end{equation}
Thus the separated origin block is exactly modeled by
\[
\mathsf H_{A_0,B_0,C_0}.
\]
At infinity the dyadic index is reflected, with parameters\linebreak[2] $(A_\infty,B_\infty,C_\infty)$.

The sequence theorem must still be transferred from shell-constant inputs to arbitrary measurable functions.  For a disjoint shell family $(E_k)$ define the conditional expectation
\begin{equation}\label{eq:conditional}
\mathsf Pg
=
\sum_k\left(\frac1{|E_k|}\int_{E_k}|g(y)|\dd y\right)\one_{E_k}.
\end{equation}
It is classical, and follows directly from Hardy--Littlewood submajorization, that
\begin{equation}\label{eq:conditional-majorization}
(\mathsf Pg)^{**}(t)\le g^{**}(t).
\end{equation}
Consequently $\mathsf P$ is bounded on every $L^{p,u}$, including weak $L^p$.  Put
\begin{equation}\label{eq:c-shell}
c_\ell=|E_\ell|^{1/p-1}\int_{E_\ell}|g(y)|\dd y,
\qquad
\mathsf Pg=J_pc.
\end{equation}
Since the separated kernel is comparable to a constant on each shell rectangle, \eqref{eq:origin-coeff-one}--\eqref{eq:origin-coeff-two} imply
\begin{equation}\label{eq:sep-shell-bound}
|E_k|^{1/q}|T_{\rm sep}g(x)|
\lesssim
(\mathsf H_{A,B,C}c)_k,
\qquad x\in E_k.
\end{equation}
Hence \cref{lem:geometric-Lorentz,thm:triangular} give
\begin{equation}\label{eq:sep-transference}
\|T_{\rm sep}g\|_{L^{q,v}}
\lesssim
\|\mathsf H_{A,B,C}c\|_{\ell^v}
\lesssim
\|c\|_{\ell^u}
\simeq
\|\mathsf Pg\|_{L^{p,u}}
\lesssim
\|g\|_{L^{p,u}}.
\end{equation}
No comparable-radius singular block is treated by this scalarization; that part is handled in \cref{sec:comparable}.

\section{Necessity of the sharp conditions}\label{sec:necessity}

We prove every condition in \cref{thm:main} by lower tests on positive subblocks of the kernel.  It is enough to use nonnegative bounded functions with compact support away from the origin.

\subsection{The local margin}

Choose a fixed transition annulus on which both broken weights and all $U$-ratios are bounded above and below.  Restricting $x$ and $y$ to two overlapping subannuli reduces \eqref{eq:main-estimate} from below to the operator in \cref{thm:local}.  Therefore
\begin{equation}\label{eq:necessity-local}
S\ge0,
\qquad
S=0\Longrightarrow u\le v.
\end{equation}
The small-ball test in that theorem also shows that no radial or angular geometry can repair $S<0$.

\subsection{The four one-sided margins}

For the origin input condition, restrict $y$ to a small ball and $x$ to a fixed transition annulus.  After writing $g=w_{\alpha_0,\alpha_\infty}f$, the lower kernel bound has the rank-one form
\begin{equation}\label{eq:origin-input-lower}
Tg(x)\gtrsim c(x)
\int_{|y|<\varepsilon}|y|^{-\alpha_0-\sigma_0}g(y)\dd y,
\end{equation}
where $c$ is bounded below on a set of positive measure.  By \cref{lem:rank-one},
\[
|y|^{-\alpha_0-\sigma_0}\one_{|y|<\varepsilon}
\in L^{p',u'}.
\]
This is equivalent to
\begin{equation}\label{eq:A0-necessary}
A_0\ge0,
\qquad
A_0=0\Longrightarrow u=1.
\end{equation}
Interchanging input and output yields
\begin{equation}\label{eq:B0-necessary}
B_0\ge0,
\qquad
B_0=0\Longrightarrow v=\infty.
\end{equation}
Sending the input to infinity while fixing the output in a transition annulus produces the density
\[
|y|^{s-d+\sigma_\infty-\alpha_\infty}\one_{|y|>4},
\]
and therefore
\begin{equation}\label{eq:Ainf-necessary}
A_\infty\ge0,
\qquad
A_\infty=0\Longrightarrow u=1.
\end{equation}
The transposed test gives
\begin{equation}\label{eq:Binf-necessary}
B_\infty\ge0,
\qquad
B_\infty=0\Longrightarrow v=\infty.
\end{equation}
Negative margins fail by a single shrinking or expanding annular packet.  At equality, the critical power $|x|^{-d/r}$ belongs only to $L^{r,\infty}$, giving the stated fine-index restrictions.  Equivalently, a logarithmic truncation over $e^{-N}<|x|<e^{-1}$ has normalized pairing growth $N^{1-1/u}$ at an input endpoint and output growth $N^{1/v}$ at an output endpoint.

\subsection{The two scale margins}

Let $E_R,F_R$ be comparable balls of radius $cR$, centered at distance comparable to $R$ from the origin.  Set
\begin{equation}\label{eq:scale-packet}
g_R=R^{-d/p}\one_{E_R},
\qquad
f_R=w_{-\alpha_0,-\alpha_\infty}g_R.
\end{equation}
For $R\downarrow0$, the lower kernel bound on $F_R\times E_R$ gives
\begin{equation}\label{eq:scale-origin-ratio}
\|w_{-\beta_0,-\beta_\infty}H^{-s/2}f_R\|_{L^{q,v}(F_R)}
\gtrsim R^{C_0},
\qquad
\|g_R\|_{L^{p,u}}\simeq1.
\end{equation}
Thus $C_0\ge0$.  For $R\uparrow\infty$ the ratio is $R^{-C_\infty}$, forcing $C_\infty\ge0$.

If $C_0=0$, choose $N$ geometrically separated shrinking packet pairs.  By \cref{lem:geometric-Lorentz} and positivity,
\[
\|g_N\|_{L^{p,u}}\simeq N^{1/u},
\qquad
\|Tg_N\|_{L^{q,v}}\gtrsim N^{1/v}.
\]
Hence $u\le v$.  Expanding packets give the same conclusion when $C_\infty=0$.  This argument does not assume $p\le q$; it is precisely why Lorentz refinement restores scale equality in the upper triangle $q<p$.

\subsection{Same-side corners}

Suppose $A_0=C_0=0$.  Then \eqref{eq:margin-identities} gives $B_0=\gamma_0>0$.  The origin triangular model contains an unweighted prefix operator.  A block of $N$ normalized input coefficients feeding one outer shell forces $u=1$, while one deep input coefficient with $N$ outer descendants forces $v=\infty$.  Therefore
\begin{equation}\label{eq:corner-A0}
A_0=C_0=0
\quad\Longrightarrow\quad
(u,v)=(1,\infty).
\end{equation}
If $B_0=C_0=0$, use the tail sum instead.  Reflection proves the two infinity corners.  Tests at the origin and infinity occur on disjoint regions, so simultaneous endpoint requirements simply accumulate; no cross condition such as $A_0+B_\infty\ge0$ appears.

For emphasis, the following elementary free example shows why the corner rule is not the intersection of the separate rules alone.

\begin{proposition}[A corner counterexample]\label{prop:corner-counterexample}
Let $d=3$, $s=1$, $H=-\Delta$, $p=q=2$, $(u,v)=(1,2)$,
\[
(\alpha_0,\beta_0)=\left(\frac32,-\frac12\right),
\qquad
(\alpha_\infty,\beta_\infty)=(1,1).
\]
Then $A_0=C_0=0$ and every other margin is positive, but \eqref{eq:main-estimate} fails.
\end{proposition}

\begin{proof}
For $0<R<1/16$, put
\[
h_R(y)=R^{-3/2}\one_{\{R<|y|<2R\}},
\qquad
f_R(y)=|y|^{-3/2}h_R(y).
\]
Then $\|h_R\|_{L^{2,1}}\simeq1$.  For $4R<|x|<1/4$, positivity and $|x-y|\simeq|x|$ give
\[
I_1f_R(x)\gtrsim |x|^{-2}.
\]
Multiplication by the output weight $|x|^{1/2}$ gives $|x|^{-3/2}$, so
\[
\big\||x|^{1/2}I_1f_R\big\|_{L^2(4R<|x|<1/4)}
\gtrsim \big(\log(1/R)\big)^{1/2}\longrightarrow\infty.
\]
\end{proof}

The preceding arguments prove the necessity of every power and fine-index condition in \cref{thm:main}.

\section{Comparable-radius blocks}\label{sec:comparable}

The singularity on comparable annuli cannot be replaced by shell averages.  We therefore prove an operator-level lemma.  At the origin define
\begin{equation}\label{eq:comparable-origin}
\mathcal C_{S,C}^{0}g(x)
=
\one_{|x|<1}
\int_{\substack{|y|<1\\ |x|/4<|y|<4|x|}}
|x|^{-S+C}|x-y|^{s-d}g(y)\dd y.
\end{equation}
At infinity define
\begin{equation}\label{eq:comparable-infinity}
\mathcal C_{S,C}^{\infty}g(x)
=
\one_{|x|>1}
\int_{\substack{|y|>1\\ |x|/4<|y|<4|x|}}
|x|^{-S-C}|x-y|^{s-d}g(y)\dd y.
\end{equation}
Changing the fixed constants $1/4$ and $4$ does not affect the statement.

\begin{lemma}[Comparable-radius Lorentz theorem]\label{lem:comparable}
Assume $S,C\ge0$.  Each operator in \eqref{eq:comparable-origin}--\eqref{eq:comparable-infinity} maps $L^{p,u}$ to $L^{q,v}$ under the following exact rules:
\begin{equation}\label{eq:comparable-rules}
\begin{cases}
\text{all }u,v,&S>0\text{ and }C>0,\\
u\le v,&S=0\text{ or }C=0.
\end{cases}
\end{equation}
Whenever $S=0$ or $C=0$, the condition $u\le v$ is necessary.
\end{lemma}

\begin{proof}
If $S=0$, then $1/q=1/p-s/d$, hence $p<q$.  Since $|x|^C\le1$ at the origin and $|x|^{-C}\le1$ at infinity, both operators are dominated by the Riesz potential $I_s$.  The global Lorentz--HLS theorem, \cref{cor:global-HLS}, gives $L^{p,u}\to L^{q,u}$ and hence every $v\ge u$.

Assume now $C=0<S$ and put
\[
\delta=d\left(\frac1p-\frac1q\right)=s-S.
\]
If $p<q$, then $0<\delta<s$.  On the comparable region $|x-y|\lesssim|x|$, and therefore
\[
|x|^{-S}|x-y|^{s-d}\lesssim |x-y|^{\delta-d}.
\]
Apply \cref{cor:global-HLS} with order $\delta$.  If $p=q$, then $S=s$ and dyadic integration around $x$ gives
\begin{equation}\label{eq:comp-maximal}
|x|^{-s}\int_{|x-y|\lesssim|x|}|x-y|^{s-d}|g(y)|\dd y
\lesssim Mg(x).
\end{equation}
If $q<p$, put
\[
\theta=d\left(\frac1q-\frac1p\right)=S-s>0.
\]
The same dyadic integration gives
\begin{equation}\label{eq:comp-uppertriangle}
\mathcal C_{S,0}g(x)\lesssim |x|^{-\theta}Mg(x).
\end{equation}
On either pure end the multiplier $|x|^{-\theta}$ belongs to $L^{d/\theta,\infty}$.  The maximal theorem in Lorentz spaces and the product estimate \eqref{eq:product} yield
\[
\||x|^{-\theta}Mg\|_{L^{q,u}}
\lesssim\|g\|_{L^{p,u}}.
\]
Thus $u\le v$ is sufficient in all three primary-exponent orders.

Finally suppose $S,C>0$.  Choose $a,b>0$ such that
\begin{equation}\label{eq:choose-ab}
\max\left\{0,d\left(\frac1q-\frac1p\right)\right\}<a+b<S,
\qquad
a<\frac d{p'},\quad b<\frac dq.
\end{equation}
Set
\begin{equation}\label{eq:PQ}
\frac1P=\frac1p+\frac ad,
\qquad
\frac1Q=\frac1q-\frac bd.
\end{equation}
Then $1<P<Q<\infty$, and $S_c=S-a-b>0$.  Choose $x_0,y_0>0$ with $x_0+y_0<C$.  At the origin factor
\begin{equation}\label{eq:comp-factor}
\mathcal C_{S,C}^{0}
=M_{|x|^{-b+y_0}}
\mathcal C_{S_c,C-x_0-y_0}^{0,c}
M_{|y|^{-a+x_0}},
\end{equation}
where the middle operator has primary exponents $P,Q$.  The identity of radial powers is
\[
(-b+y_0)+[-S_c+(C-x_0-y_0)]+(-a+x_0)=-S+C.
\]
Both multipliers are strictly better than their critical powers and therefore belong to the required Lorentz classes with arbitrary fine index.  By choosing their fine indices according to whether $u\le P$ and $Q\le v$, the product inequality gives
\[
L^{p,u}
\xrightarrow{M_{\rm in}}L^P
\xrightarrow{\mathcal C^{0,c}}L^Q
\xrightarrow{M_{\rm out}}L^{q,v}.
\]
The middle strong estimate follows from a dyadic shell decomposition: finite shell overlap and the local strong estimate give a diagonal coefficient $2^{k(C-x_0-y_0)}$; because $P<Q$ and the exponent is positive, ordinary $\ell^P\hookrightarrow\ell^Q$ and geometric summation apply.  At infinity one uses multiplier powers $-a-x_0$ and $-b-y_0$; the proof is identical after radial reflection.

Necessity of $u\le v$ at $S=0$ follows from the local packets in \cref{thm:local}; at $C=0$ it follows from geometrically separated same-scale packets as in the scale-margin test.  This completes the proof.
\end{proof}

\section{Nine-block sufficiency}\label{sec:sufficiency}

Choose a smooth radial partition
\begin{equation}\label{eq:partition}
\eta_0+\eta_m+\eta_\infty=1
\end{equation}
with
\[
\begin{aligned}
\supp\eta_0&\subset\{|x|<3/4\},
&\supp\eta_m&\subset\{1/2<|x|<2\},\\
\supp\eta_\infty&\subset\{|x|>3/2\}.
\end{aligned}
\]
After conjugating by the input and output broken weights, write
\[
T_{ij}=\eta_iw_{-\beta_0,-\beta_\infty}H^{-s/2}
 w_{-\alpha_0,-\alpha_\infty}\eta_j,
\qquad i,j\in\{0,m,\infty\}.
\]
Rows denote output zones and columns input zones.

\subsection{The two deep diagonal blocks}

For $T_{00}$ split into comparable radii and separated radii.  The comparable piece is controlled by \cref{lem:comparable} with parameters $(S,C_0)$.  The separated piece is controlled by \cref{thm:triangular} with $(A_0,B_0,C_0)$ and the conditional-expectation transference \eqref{eq:sep-transference}.  Hence the exact origin truth table is
\begin{center}
\begin{tabular}{cccc}
\toprule
$A_0$ & $B_0$ & $C_0$ & condition\\
\midrule
$>0$ & $>0$ & $>0$ & none\\
$=0$ & $>0$ & $>0$ & $u=1$\\
$>0$ & $=0$ & $>0$ & $v=\infty$\\
$>0$ & $>0$ & $=0$ & $u\le v$\\
$=0$ & $>0$ & $=0$ & $(u,v)=(1,\infty)$\\
$>0$ & $=0$ & $=0$ & $(u,v)=(1,\infty)$\\
\bottomrule
\end{tabular}
\end{center}
together with the global local rule $S=0\Rightarrow u\le v$.  The row $A_0=B_0=0$ is impossible under $C_0\ge0$ by \eqref{eq:margin-identities}.  Radial reflection gives the identical assertion for $T_{\infty\infty}$ with $(A_\infty,B_\infty,C_\infty)$.

\subsection{Transition and cross-tail blocks}

The nine blocks require the following conditions:
\begin{center}
\begin{tabular}{c|ccc}
\toprule
output $\backslash$ input & $0$ & $m$ & $\infty$\\
\midrule
$0$ & deep $(A_0,B_0,C_0)$ & local / $B_0$ & rank one: $B_0,A_\infty$\\
$m$ & local / $A_0$ & local Riesz & local / $A_\infty$\\
$\infty$ & rank one: $B_\infty,A_0$ & local / $B_\infty$ & deep $(A_\infty,B_\infty,C_\infty)$\\
\bottomrule
\end{tabular}
\end{center}
For each block involving the middle annulus, insert a cutoff separating $|x-y|\le1/8$ from $|x-y|>1/8$.  On the near part both variables remain in a fixed annulus; the broken weights, $U$, and every ratio in \eqref{eq:clean-kernel-main} are bounded above and below, so \cref{thm:local} applies.  On the far part there is no diagonal singularity.  If the other variable remains in a compact annulus, a bounded-kernel estimate applies.  If it tends to zero or infinity, the kernel is one of the four rank-one factors in \eqref{eq:rank-factors}, and \cref{lem:rank-one} gives exactly the displayed $A_j$ or $B_j$ rule.

The genuinely mixed blocks are products up to fixed constants:
\begin{equation}\label{eq:mixed-products}
T_{0\infty}: b_0(x)a_\infty(y),
\qquad
T_{\infty0}: b_\infty(x)a_0(y).
\end{equation}
Thus no cross-sum condition involving, for example, $\alpha_0+\beta_\infty$ can arise.

\begin{proof}[Proof of sufficiency in \cref{thm:main}]
Assume \eqref{eq:power-all} and all applicable fine-index rules.  The local theorem controls every near-diagonal transition piece.  The rank-one lemma controls every separated transition and cross-tail block.  The triangular theorem and \cref{lem:comparable} control the separated and comparable parts of the two deep diagonal blocks, including every same-side corner.  Since there are only nine blocks, summing their estimates proves \eqref{eq:main-estimate} for bounded compactly supported inputs away from the origin.  The extension to all measurable inputs is given in \cref{sec:extension}.  Necessity was proved in \cref{sec:necessity}, completing the proof of \cref{thm:main}.
\end{proof}

\section{Extension, duality, and consistency checks}\label{sec:extension}

\subsection{Weak Lorentz inputs}

Simple compactly supported functions are not norm dense in $L^{p,\infty}$, so a completion argument is insufficient at $u=\infty$.  For a nonnegative $g\in L^{p,u}$ define
\begin{equation}\label{eq:monotone-truncation}
g_n(x)=\min\{g(x),n\}\one_{\{1/n<|x|<n\}}.
\end{equation}
Then $g_n\uparrow g$.  Put $f_n=w_{-\alpha_0,-\alpha_\infty}g_n$ and define
\begin{equation}\label{eq:positive-extension}
\mathcal Tg(x)=\lim_{n\to\infty}
 w_{-\beta_0,-\beta_\infty}(x)H^{-s/2}f_n(x)
\in[0,\infty].
\end{equation}
The estimates for $g_n$, the Fatou property of $L^{q,v}$, and
$\|g_n\|_{L^{p,u}}\le\|g\|_{L^{p,u}}$ give
\[
\|\mathcal Tg\|_{L^{q,v}}\le C\|g\|_{L^{p,u}}.
\]
Hence the limit is finite almost everywhere.  For signed or complex $g$, apply the construction to the positive and negative parts of the real and imaginary components.  Positivity gives $|\mathcal Tg|\le\mathcal T|g|$, so the integral is absolutely convergent almost everywhere.  This is the canonical extension in \cref{thm:main}.

\subsection{Duality}

The symmetry $K_s^H(x,y)=K_s^H(y,x)$ induces the transformation
\begin{equation}\label{eq:duality-map}
(p,q,u,v,\alpha_j,\beta_j)
\longmapsto
(q',p',v',u',\beta_j,\alpha_j).
\end{equation}
Under \eqref{eq:duality-map}, $S$ and the two $C_j$ are fixed, while $A_j$ and $B_j$ are interchanged.  The condition $u\le v$ becomes $v'\le u'$, $u=1$ becomes a dual output fine index $\infty$, and $v=\infty$ becomes a dual input fine index $1$.  Thus every rule in \cref{thm:main} is invariant.  At weak fine indices this uses the \Kothe{} associate pairing, not an identification of the full Banach dual of weak $L^p$.

\subsection{Strong spaces}

Set $u=p$ and $v=q$.  Since $p>1$, an input equality requiring $u=1$ is impossible; since $q<\infty$, an output equality requiring $v=\infty$ is impossible.  Hence all four one-sided inequalities are strict.  A scale equality requires $p\le q$.  Consequently the strong estimate holds exactly under
\[
S\ge0,
\]
\[
\alpha_0+\sigma_0<\frac d{p'},
\qquad
\beta_0+\sigma_0<\frac dq,
\]
\[
\alpha_\infty>s+\sigma_\infty-\frac dp,
\qquad
\beta_\infty>s+\sigma_\infty-\frac d{q'},
\]
and
\[
\alpha_0+\beta_0\le S\le\alpha_\infty+\beta_\infty
\quad(p\le q),
\]
while both scale inequalities are strict for $q<p$.

\subsection{Free and homogeneous models}

For $H=-\Delta$, $\sigma_0=\sigma_\infty=0$, and all four weight exponents zero, the inequalities $C_0,C_\infty\ge0$ force $S=0$, and the only fine-index condition is $u\le v$.  Thus \cref{thm:main} reduces exactly to the classical Lorentz Hardy--Littlewood--Sobolev theorem.

Suppose instead that $U(r)=r^{-\sigma}$ globally and
\[
\alpha_0=\alpha_\infty=\alpha,
\qquad
\beta_0=\beta_\infty=\beta.
\]
Then $C_0,C_\infty\ge0$ force
\[
\alpha+\beta=S.
\]
Writing
\[
A=\frac d{p'}-\alpha-\sigma,
\qquad
B=\frac dq-\beta-\sigma,
\]
the exact homogeneous theorem is
\[
S\ge0,
\qquad
\alpha+\beta=S,
\qquad
A,B\ge0.
\]
If $A,B>0$, the fine-index condition is $u\le v$; if either $A=0$ or $B=0$, the only pair is $(1,\infty)$.  This range can be nonempty for $q<p$: the strong estimate fails there because $u=p>q=v$, not because homogeneous scale equality is intrinsically impossible.

\subsection{Signed exponents, thin balls, and angular multiplicity}

All deep-block computations use $\sigma_j$ only through the margins and
\[
\gamma_j=d-s-2\sigma_j>0.
\]
A negative $\sigma_j$ changes whether a radial factor grows or decays but creates no new fine-index rule or unshifted ceiling on $\alpha_j$ or $\beta_j$.

A ball of radius $h\ll R$ centered at radius $R$ has normalized packet ratio
\begin{equation}\label{eq:thin-ball}
R^{C_0}\left(\frac hR\right)^S
\quad\text{at the origin},
\qquad
R^{-C_\infty}\left(\frac hR\right)^S
\quad\text{at infinity}.
\end{equation}
Because $S,C_0,C_\infty\ge0$, optimizing independently in $R$ and $h/R$ creates no additional condition.  Packing many equal balls at one scale is already governed by the local Lorentz theorem, while packing geometric scales is governed by the annular theorem.  Angular multiplicity therefore adds no inequality.

\subsection{Boundary diagnostics}

The proof also records the mechanism of every excluded case:
\begin{center}
\begin{tabular}{lll}
\toprule
failed condition & test & divergence\\
\midrule
$S<0$ & thin transition ball & $\varepsilon^S$\\
$S=0$, $u>v$ & shrinking local packets & $N^{1/v-1/u}$\\
$A_j<0$ & one input packet & geometric power\\
$A_j=0$, $u>1$ & logarithmic input & $N^{1-1/u}$\\
$B_j<0$ & one output packet & geometric power\\
$B_j=0$, $v<\infty$ & critical output power & $N^{1/v}$\\
$C_j<0$ & one same-scale packet & geometric power\\
$C_j=0$, $u>v$ & $N$ scale copies & $N^{1/v-1/u}$\\
$A_j=C_j=0$, $v<\infty$ & prefix descendants & $N^{1/v}$\\
$B_j=C_j=0$, $u>1$ & tail concentration & $N^{1-1/u}$\\
\bottomrule
\end{tabular}
\end{center}
These tests also show that constants necessarily diverge when a positive strict margin approaches zero at an incompatible fine-index pair.

\subsection{Further directions}

The present theorem concerns the pure-power $A_2$ branch and the clean fractional-order range.  Natural extensions include logarithmic ground-state branches, the critical orders $s=d$ and $s=d-2\sigma_j$ where the kernel acquires logarithmic or saturated terms, endpoint primary exponents $p=1$ or $q=\infty$, and nonradial perturbations for which the two-regime ground-state geometry is unavailable.  Each requires a different kernel model and is not claimed here.

\bibliographystyle{amsplain}
\bibliography{references}

\providecommand{\bysame}{\leavevmode\hbox to3em{\hrulefill}\thinspace}
\providecommand{\MR}{\relax\ifhmode\unskip\space\fi MR }
\providecommand{\MRhref}[2]{%
  \href{http://www.ams.org/mathscinet-getitem?mr=#1}{#2}
}
\providecommand{\href}[2]{#2}
\begin{thebibliography}{10}

\bibitem{AdamsHedberg1996}
David~R. Adams and Lars~Inge Hedberg, \emph{Function spaces and potential
  theory}, Grundlehren der mathematischen Wissenschaften, vol. 314,
  Springer-Verlag, Berlin, 1996.

\bibitem{BennettSharpley1988}
Colin Bennett and Robert Sharpley, \emph{Interpolation of operators}, Pure and
  Applied Mathematics, vol. 129, Academic Press, Boston, MA, 1988.

\bibitem{BerghLofstrom1976}
J{\"o}ran Bergh and J{\"o}rgen L{\"o}fstr{\"o}m, \emph{Interpolation spaces: An
  introduction}, Grundlehren der mathematischen Wissenschaften, vol. 223,
  Springer-Verlag, Berlin-New York, 1976.

\bibitem{BongioanniHarboureSalinas2008}
Bruno Bongioanni, Eleonor Harboure, and Oscar Salinas, \emph{Weighted
  inequalities for negative powers of {Schr\"odinger} operators}, J. Math.
  Anal. Appl. \textbf{348} (2008), no.~1, 12--27.

\bibitem{BuiEtAl2017}
The~Anh Bui, Piero D'Ancona, Xuan~Thinh Duong, Ji~Li, and Fu~Ken Ly,
  \emph{Weighted estimates for powers and smoothing estimates of
  {Schr\"odinger} operators with inverse-square potentials}, J. Differential
  Equations \textbf{262} (2017), no.~3, 2771--2807.

\bibitem{CaffarelliKohnNirenberg1984}
Luis Caffarelli, Robert Kohn, and Louis Nirenberg, \emph{First order
  interpolation inequalities with weights}, Compositio Math. \textbf{53}
  (1984), no.~3, 259--275.

\bibitem{CascanteOrtegaVerbitsky2004}
Carme Cascante, Joaquim~M. Ortega, and Igor~E. Verbitsky, \emph{Nonlinear
  potentials and two weight trace inequalities for general dyadic and radial
  kernels}, Indiana Univ. Math. J. \textbf{53} (2004), no.~3, 845--882.

\bibitem{Davies1989}
E.~Brian Davies, \emph{Heat kernels and spectral theory}, Cambridge Tracts in
  Mathematics, vol.~92, Cambridge University Press, Cambridge, 1989.

\bibitem{Duoandikoetxea2013}
Javier Duoandikoetxea, \emph{Fractional integrals on radial functions with
  applications to weighted inequalities}, Ann. Mat. Pura Appl. (4) \textbf{192}
  (2013), no.~4, 553--568.

\bibitem{FabesKenigSerapioni1982}
Eugene~B. Fabes, Carlos~E. Kenig, and Raul~P. Serapioni, \emph{The local
  regularity of solutions of degenerate elliptic equations}, Comm. Partial
  Differential Equations \textbf{7} (1982), no.~1, 77--116.

\bibitem{GorbachevIvanov2019}
Dmitry~V. Gorbachev and Valerii~I. Ivanov, \emph{Weighted inequalities for
  {Dunkl--Riesz} potential}, Chebyshevskii Sb. \textbf{20} (2019), no.~1,
  131--147, In Russian.

\bibitem{Grafakos2014}
Loukas Grafakos, \emph{Classical {Fourier} analysis}, third ed., Graduate Texts
  in Mathematics, vol. 249, Springer, New York, 2014.

\bibitem{HanninenHytonenLi2016}
Timo~S. H{\"a}nninen, Tuomas~P. Hyt{\"o}nen, and Kangwei Li, \emph{Two-weight
  {$L^p$--$L^q$} bounds for positive dyadic operators: unified approach to
  {$p\le q$} and {$p>q$}}, Potential Anal. \textbf{45} (2016), no.~3, 579--608.

\bibitem{HardyLittlewood1928}
G.~H. Hardy and J.~E. Littlewood, \emph{Some properties of fractional
  integrals. {I}}, Math. Z. \textbf{27} (1928), no.~1, 565--606.

\bibitem{HuZahle2009}
Jiaxin Hu and Martina Z{\"a}hle, \emph{Generalized {Bessel} and {Riesz}
  potentials on metric measure spaces}, Potential Anal. \textbf{30} (2009),
  no.~4, 315--340.

\bibitem{Hunt1966}
Richard~A. Hunt, \emph{On {$L(p,q)$} spaces}, Enseign. Math. (2) \textbf{12}
  (1966), 249--276.

\bibitem{IKO2017}
Kazuhiro Ishige, Yoshitsugu Kabeya, and El~Maati Ouhabaz, \emph{The heat kernel
  of a {Schr\"odinger} operator with inverse square potential}, Proc. Lond.
  Math. Soc. (3) \textbf{115} (2017), no.~2, 381--410.

\bibitem{IshigeTateishi2020a}
Kazuhiro Ishige and Yujiro Tateishi, \emph{Decay estimates for {Schr\"odinger}
  heat semigroup with inverse square potential in {Lorentz} spaces}, J. Evol.
  Equ. \textbf{22} (2022), Paper No. 16, 25 pp.

\bibitem{IshigeTateishi2021}
\bysame, \emph{Decay estimates for {Schr\"odinger} heat semigroup with inverse
  square potential in {Lorentz} spaces {II}}, Discrete Contin. Dyn. Syst.
  \textbf{42} (2022), no.~1, 369--401.

\bibitem{Kairema2013}
Anna Kairema, \emph{Two-weight norm inequalities for potential type and maximal
  operators in a metric space}, Publ. Mat. \textbf{57} (2013), no.~1, 3--56.

\bibitem{Kerman1983}
Richard~A. Kerman, \emph{Convolution theorems with weights}, Trans. Amer. Math.
  Soc. \textbf{280} (1983), no.~1, 207--219.

\bibitem{KMMVZZ2018}
Rowan Killip, Changxing Miao, Monica Visan, Jian Zhang, and Jiqiang Zheng,
  \emph{Sobolev spaces adapted to the {Schr\"odinger} operator with
  inverse-square potential}, Math. Z. \textbf{288} (2018), no.~3--4,
  1273--1298.

\bibitem{LiuYuZhou2026}
Haochen Liu, Qinghao Yu, and Hongyan Zhou, \emph{Sharp and endpoint two-weight
  fractional integral estimates for {Schr\"odinger} operators with
  inverse-square potentials}, arXiv preprint, 2026, arXiv:2607.09585.

\bibitem{Lorentz1950}
George~G. Lorentz, \emph{Some new functional spaces}, Ann. of Math. (2)
  \textbf{51} (1950), no.~1, 37--55.

\bibitem{Muckenhoupt1972}
Benjamin Muckenhoupt, \emph{Weighted norm inequalities for the {Hardy} maximal
  function}, Trans. Amer. Math. Soc. \textbf{165} (1972), 207--226.

\bibitem{MuckenhouptWheeden1974}
Benjamin Muckenhoupt and Richard~L. Wheeden, \emph{Weighted norm inequalities
  for fractional integrals}, Trans. Amer. Math. Soc. \textbf{192} (1974),
  261--274.

\bibitem{NowakStempak2017}
Adam Nowak and Krzysztof Stempak, \emph{Potential operators associated with
  {Hankel} and {Hankel--Dunkl} transforms}, J. Anal. Math. \textbf{131} (2017),
  277--321.

\bibitem{ONeil1963}
Richard O'Neil, \emph{Convolution operators and {$L(p,q)$} spaces}, Duke Math.
  J. \textbf{30} (1963), no.~1, 129--142.

\bibitem{Ouhabaz2005}
El~Maati Ouhabaz, \emph{Analysis of heat equations on domains}, London
  Mathematical Society Monographs Series, vol.~31, Princeton University Press,
  Princeton, NJ, 2005.

\bibitem{Sawyer1988}
Eric~T. Sawyer, \emph{A characterization of two weight norm inequalities for
  fractional and {Poisson} integrals}, Trans. Amer. Math. Soc. \textbf{308}
  (1988), no.~2, 533--545.

\bibitem{SawyerWheeden1992}
Eric~T. Sawyer and Richard~L. Wheeden, \emph{Weighted inequalities for
  fractional integrals on {Euclidean} and homogeneous spaces}, Amer. J. Math.
  \textbf{114} (1992), no.~4, 813--874.

\bibitem{Shen1995}
Zhongwei Shen, \emph{{$L^p$} estimates for {Schr\"odinger} operators with
  certain potentials}, Ann. Inst. Fourier (Grenoble) \textbf{45} (1995), no.~2,
  513--546.

\bibitem{Sobolev1938}
Sergei~L. Sobolev, \emph{On a theorem of functional analysis}, Mat. Sb.
  \textbf{4 (46)} (1938), 471--497, English translation: Amer. Math. Soc.
  Transl. (2) 34 (1963), 39--68.

\bibitem{SteinWeiss1958}
Elias~M. Stein and Guido Weiss, \emph{Fractional integrals on {$n$}-dimensional
  {Euclidean} space}, J. Math. Mech. \textbf{7} (1958), no.~4, 503--514.

\bibitem{Tanaka2014}
Hitoshi Tanaka, \emph{A characterization of two-weight trace inequalities for
  positive dyadic operators in the upper triangle case}, Potential Anal.
  \textbf{41} (2014), no.~2, 487--499.

\end{thebibliography}

\end{document}